\documentclass[english]{amsart}

\usepackage{esint}
\usepackage[svgnames]{xcolor}
\usepackage[colorlinks,citecolor=red,pagebackref,hypertexnames=false,breaklinks]{hyperref}
\usepackage{pgf,tikz}
\usepackage{pdfsync}

\usepackage{dsfont}
\usepackage{url}
\usepackage[utf8]{inputenc}
\usepackage[T1]{fontenc}
\usepackage{lmodern}
\usepackage{babel}
\usepackage{mathtools}
\usepackage{amssymb}
\usepackage{lipsum}
\usepackage{mathrsfs}
\usepackage{color}
\usepackage[skip=4pt plus1pt, indent=12pt]{parskip}

\newtheorem{theorem}{Theorem}[section]
\newtheorem{proposition}{Proposition}[section]
\newtheorem{lemma}{Lemma}[section]

\newtheorem{corollary}{Corollary}[section]

\newtheorem{remark}{Remark}[section]

\numberwithin{equation}{section}

\title[Semilinear elliptic inverse problem]{H\"older stability for a  semilinear elliptic inverse problem}

\author[Mourad Choulli]{Mourad Choulli}
\address{Universit\'e de Lorraine,  France}
\email{mourad.choulli@univ-lorraine.fr}

\thanks{The author is supported by the grant ANR-17-CE40-0029 of the French National Research Agency ANR (project MultiOnde). 
}

\date{}

\begin{document}

\begin{abstract}
We are  concerned with the problem of determining the nonlinear term in a semilinear elliptic equation by boundary measurements. Precisely, we improve \cite[Theorem 1.3]{CHY}, where a logarithmic type stability estimate was proved. We show actually that we have a H\"older stability estimate with less boundary measurents and less regular nonlinearities. We establish our stability inequality by following the same method as  in \cite{Ch22}. This method consists in constructing special solutions vanishing on a subboundary of the domain.
\end{abstract}



\maketitle


\section{Introduction}\label{section1}

Let $\Omega$ be a $C^{1,1}$ bounded domain of $\mathbb{R}^n$ ($n\geq 3$) with boundary $\Gamma$. Fix $A=(a^{ij})$ a symmetric $n\times n$ matrix satisfying
\begin{equation}\label{e1.1}
\kappa |\xi|^2\le A\xi \cdot \xi ,\; \xi \in \mathbb{R}^n, \quad \max_{i,j}|a^{ij}|\le \kappa^{-1}
\end{equation}
where $0<\kappa <1$ is a given constant.

Pick $\alpha \ge 0$, $\mu_j>0$, $j=1,2$, and $0\le \mathfrak{c}<\kappa \lambda_1$, where $\lambda_1$ denotes the first eigenvalue of the Laplace operator on $\Omega$ under Dirichlet boundary condition. Consider further the assumption 

\noindent
\textbf{(a1)} $a\in C^1(\mathbb{R},\mathbb{R})$ and
\[
|a(z)|\le \mu_1 +\mu_2 |z|^\alpha,\quad a'(z)\ge -\mathfrak{c} ,\quad z\in \mathbb{R},
\]
and the BVP
\begin{equation}\label{e1.3}
\left\{
\begin{array}{ll}
-\mathrm{div}(A\nabla u) +a(u)=0\quad \mbox{in}\; \Omega ,
\\
u_{|\Gamma}=f .
\end{array}
\right.
\end{equation}

Let $f\in H^{1/2}(\Gamma)$. We say that $u\in H^1(\Omega)$ is a weak solution of the BVP \eqref{e1.3} if $u_{|\Gamma}=f$ in the trace sense and if
\[
\int_\Omega A\nabla u\cdot\nabla vdx+\int_\Omega a(u)vdx=0,\quad v \in H_0^1(\Omega ).
\]

Throughout this text, the ball of a normed space $N$ with center $0$ and radius $m>0$ will denoted  by $B_N(m)$.

Assume that $\alpha <(n+2)/(n-2)$ and that $a$ satisfies \textbf{(a1)}. Let $m>0$. By slight modifications of the proof \cite[Theorem 2.1]{CHY} we get that, for any $f\in H^{1/2}(\Gamma)$, the BVP \eqref{e1.3} admits a unique weak solution $u_a(f)\in H^1(\Omega)$. Furthermore, the following estimate holds
\begin{equation}\label{e1.4}
\|u_a(f)\|_{H^1(\Omega)}\le C(1+m^\alpha),\quad f\in B_{H^{1/2}(\Gamma)}(m).
\end{equation}
Here and henceforth $C=C(n,\Omega,\kappa, \alpha ,\mathfrak{c},\mu_1,\mu_2)$ denotes a generic constant.

We endow in the sequel $H^{1/2}(\Gamma)$ (identified with the quotient space $H^1(\Omega)/H_0^1(\Omega)$) with the quotient norm
\[
\|\varphi\|_{H^{1/2}(\Gamma)}=\min\left\{ \|v\|_{H^1(\Omega)};\; v\in \dot{\varphi}\right\},\quad \varphi \in H^{1/2}(\Gamma),
\]
where 
\[
\dot{\varphi}=\left\{v\in H^1(\Omega);\; v_{|\Gamma}=\varphi\right\}.
\]

We associate to $a$ the Dirichlet-to-Neumann map 
\[
\Lambda_a :H^{1/2}(\Gamma)\rightarrow H^{-1/2}(\Gamma), 
\]
defined by
\[
\langle \Lambda_a(f) ,\varphi\rangle =\int_\Omega A\nabla u_a(f)\cdot \nabla vdx+\int_\Omega a(u_a(f))vdx,\quad v\in \dot{\varphi}.
\] 
As $u_a(f)$ is the weak solution of the BVP \eqref{e1.3}, we easily check that the right hand side of the inequality above is independent of $v$, $v\in \dot{\varphi}$. On the other hand, we can mimic the proof of \cite[Lemma 2.1]{CHY} in order to obtain
\[
\left| \int_\Omega a(u_a(f))vdx \right|\le C\left(1+\|u_a(f)\|_{H^1(\Omega)}^\alpha\right)\|v\|_{H^1(\Omega)}.
\]
This and \eqref{e1.4} show that 
\[
\| \Lambda_a (f)\|_{H^{-1/2}(\Gamma)}\le C(1+m^\alpha)^\alpha,\quad f\in B_{H^{1/2}(\Gamma)}(m).
\]

For $n\ge 4$, fix $n/2<p<n$. Set $q_n=2n/(n-4)$ if $n>4$ and $q_4=2r/(2-r)$ for some arbitrary fixed $1\le r<2$.  Define
\begin{align*}
&\alpha_3=3,\qquad \alpha_n=q_n/p,\quad n\ge 4 ,
\\
&\beta_3=1/2,\qquad \beta_n=2-n/p, \quad n\ge 4,
\end{align*}
and
\begin{align*}
&\mathfrak{X}_3=H^2(\Omega)\cap C^{0,\beta_3}(\overline{\Omega}),\hskip 1.3cm \mathfrak{Z}_3=H^{3/2}(\Gamma),
\\
&\mathfrak{X}_n=W^{2,p}(\Omega)\cap C^{0,\beta_n}(\overline{\Omega}),\hskip 1cm \mathfrak{Z}_n=W^{2-1/p,p}(\Gamma),\quad n\ge 4.
\end{align*}

Henceforth, $\Gamma_0$ will denote a nonempty open subset of $\Gamma$ and $\Gamma_0\Subset \Gamma_1\subset \Gamma$. Define $H_{\Gamma_0}^{1/2}(\Gamma)$ as follows
\[
H_{\Gamma_0}^{1/2}(\Gamma)=\{f\in H^{1/2}(\Gamma);\; \mathrm{supp}(f)\subset \overline{\Gamma_0}\}.
\]
The (closed) subspace $H_{\Gamma_0}^{1/2}(\Gamma)$ will be equipped with the norm of $H^{1/2}(\Gamma)$.

Next, fix $\chi\in C_0^\infty (\Gamma_1)$ so that $\chi=1$ in $\overline{\Gamma}_0$. If $\psi \in H^{-1/2}(\Gamma)$ we define $\chi\psi$ by 
\[
\langle \chi \psi ,\varphi\rangle_{1/2} = \langle \psi ,\chi \varphi\rangle_{1/2},\quad \varphi\in H^{1/2}(\Gamma),
\]
where $\langle \cdot,\cdot \rangle_{1/2}$ is the duality pairing between $H^{1/2}(\Gamma)$ and its dual $H^{-1/2}(\Gamma)$.

Clearly, $\chi \psi\in H^{-1/2}(\Gamma)$, $\mathrm{supp}(\chi \psi)\subset \Gamma_1$ and the following identity holds
\begin{equation}\label{e1.5}
\langle \chi \psi ,\varphi\rangle_{1/2} = \langle \psi , \varphi\rangle_{1/2},\quad \varphi\in H_{\Gamma_0}^{1/2}(\Gamma).
\end{equation}


Consider the closed subset of  $\mathfrak{Z}_n$
\[
\mathfrak{Z}_n^0=\{f\in \mathfrak{Z}_n;\; \mathrm{supp}(f)\subset \overline{\Gamma_0}\}, 
\]
and define the partial Dirichlet-to-Neumann map
\[
\tilde{\Lambda}_a: f\in \mathfrak{Z}_n^0\mapsto \chi\Lambda_a(f)\in H^{-1/2}(\Gamma).
\]

Let $\Theta$ be the vector space of functions $\Lambda: \mathfrak{Z}_n^0\rightarrow H^{-1/2}(\Gamma)$ that are everywhere Fr\'echet differentiable such that the Fr\'echet differential of $\Lambda$ at $f\in \mathfrak{Z}_n^0$ has a unique extension denoted by $d\Lambda (f)\in \mathscr{B}(H_{\Gamma_0}^{1/2}(\Gamma),H^{-1/2}(\Gamma))$ and such that for every $m>0$ we have
\[
\mathfrak{p}_m(\Lambda)=\sup_{f\in B_{\mathfrak{Z}_n^0}(m)}\left (\|\Lambda (f)\|_{H^{-1/2}(\Gamma)}+\|d\Lambda (f)\|_{\mathrm{op}}\right)<\infty.
\] 
Here and henceforth, $\|\cdot \|_{\mathrm{op}}$ denotes the usual norm of $\mathscr{B}(H_{\Gamma_0}^{1/2}(\Gamma),H^{-1/2}(\Gamma))$.

Observe that $(\mathfrak{p}_m)_{m>0}$ defines a family of semi-norms on $\Theta$. We can then endow $\Theta$ with the topology induced by this family of semi-norms.

Pick a non decreasing function $\gamma:\mathbb{R}\rightarrow (0,\infty)$   and consider the assumption

\noindent 
\textbf{(a2)} $a\in C^1(\mathbb{R})$ satisfies $|a'(z)|\le \gamma(|z|)$, $z\in \mathbb{R}$.

Under the assumption that $a$ satisfies both \textbf{(a1)} and \textbf{(a2)} with $\alpha\le \alpha_n$, we easily derive from \eqref{e2.1} and \eqref{e2.8}  that $\tilde{\Lambda}_a\in \Theta$.

In what follows 
\begin{align*}
&\mathcal{C}_3=\mathcal{C}_3(n,\Omega,\kappa, \mathfrak{c},\alpha, \mu_1,\mu_2),
\\
&\mathcal{C}_4=\mathcal{C}_4(n,\Omega,\kappa, \mathfrak{c},\alpha, \mu_1,\mu_2, p, r),
\\
&\mathcal{C}_n=\mathcal{C}_n(n,\Omega,\kappa, \mathfrak{c},\alpha, \mu_1,\mu_2, p),\quad n>4,
\\
&C_m=C_m(\mathcal{C}_n,m,\gamma(\varrho_m))
\end{align*}
denote generic constants.

For sake of clarity we first state a H\"older stability result for the problem of determining $a$ from its corresponding partial Dirichlet to Neumann map $\tilde{\Lambda}_a$. A more general result will be given below.

It is worth noticing that, contrary to the Dirichlet-to-Neumann map associated to a Schr\"odinger equation which is a linear map, $\Lambda_a$ is a nonlinear map. Moreover, as we saw above, $\Lambda_a$ does not belong to a normed vector space.

\begin{theorem}\label{theorem1.0}
Assume that $\alpha \le \alpha_n$ and let $a_1,a_2$ satisfy both \textbf{(a1)} and \textbf{(a2)}. For every $\tau >0$, there exists $m=m(\tau)>0$ such that
\begin{equation}\label{e1.6.0}
\|a'_1-a'_2\|_{C([-\tau ,\tau])}\le C_m\mathfrak{p}_m\left( \tilde{\Lambda}_{a_1}-\tilde{\Lambda}_{a_2}\right)^{\beta_n/(2+\beta_n)}.
\end{equation}
\end{theorem}

Observe that, with the aid of the mean value theorem, we easily derive from \eqref{e1.6.0} the following estimate
\[
\|a_1-a_2\|_{C([-\tau ,\tau])}\le |a_1(0)-a_2(0)|+\tau C_m\mathfrak{p}_m\left( \tilde{\Lambda}_{a_1}-\tilde{\Lambda}_{a_2}\right)^{\beta_n/(2+\beta_n)}.
\]
This estimate yields in a straightforward manner the following uniqueness result.

\begin{corollary}\label{corollary1.0}
Suppose that $\alpha\le \alpha_n$ and let $a_1,a_2$ satisfy \textbf{(a1)} and \textbf{(a2)} together with $a_1(0)=a_2(0)$. If $\tilde{\Lambda}_{a_1}=\tilde{\Lambda}_{a_2}$ then $a_1=a_2$.
\end{corollary}

As we already mentioned, we prove a result implying Theorem \ref{theorem1.0}. Before giving the statement of this result we need to introduce new definitions and notations.

Fix $x_\ast\in \Gamma_0$ and $h\in \mathfrak{Z}_n^0$ satisfying $h(x_\ast)=1$, and, for every $t\in \mathbb{R}$, let $\mathfrak{f}_t=th$.

Define the family of partial Dirichlet-to-Neumann maps $(\tilde{\Lambda}_a^t)_{t\in \mathbb{R}}$ as follows
\[
\tilde{\Lambda}_a^t : f\in \mathfrak{Z}_n^0\mapsto \chi\Lambda_a(\mathfrak{f}_t+f)\in H^{-1/2}(\Gamma),\quad t\in \mathbb{R}.
\]
From a result in Section \ref{section3}, the Fréchet differential of $\tilde{\Lambda}_a^t$ at $0$ has a bounded extension $d\tilde{\Lambda}_a^t(0)\in \mathscr{B}(H_{\Gamma_0}^{1/2}(\Gamma),H^{-1/2}(\Gamma))$.

Our main goal in this work is to prove the following theorem, where
\[
m_n^\tau=\max_{|t|\le \tau}\|\mathfrak{f}_t\|_{\mathfrak{Z}_n}=\tau\|h\|_{\mathfrak{Z}_n^0},\quad \tau >0.
\]

\begin{theorem}\label{theorem1.1}
Assume that $\alpha \le \alpha_n$. Let $a_1,a_2$ satisfying \textbf{(a1)} and \textbf{(a2)}. For each $\tau >0$, we have
\begin{equation}\label{e1.6}
\|a'_1-a'_2\|_{C([-\tau ,\tau])}\le C_{m_n^\tau}\sup_{|t|\le \tau}\|d\tilde{\Lambda}^t_{a_1}(0)-d\tilde{\Lambda}^t_{a_2}(0)\|_{\mathrm{op}}^{\beta_n/(2+\beta_n)}.
\end{equation}
\end{theorem}

Note that inequality \eqref{e1.6} can be rewritten in term of $\tilde{\Lambda}_{a_j}$, $j=1,2$. Precisely, we have
\[
\|a'_1-a'_2\|_{C([-\tau ,\tau])}\le C_{m_n^\tau}\sup_{|t|\le \tau}\|d\tilde{\Lambda}_{a_1}(\mathfrak{f}_t)-d\tilde{\Lambda}_{a_2}(\mathfrak{f}_t)\|_{\mathrm{op}}^{\beta_n/(2+\beta_n)}.
\]

Also, as for Theorem \ref{theorem1.0}, for each $\tau>0$, \eqref{e1.6} implies
\begin{align*}
\|a_1-a_2\|_{C([-\tau ,\tau])}\le |a_1(0)&-a_2(0)|
\\
&+\tau C_{m_n^\tau}\sup_{|t|\le \tau}\|d\tilde{\Lambda}^t_{a_1}(0)-d\tilde{\Lambda}^t_{a_2}(0)\|_{\mathrm{op}}^{\beta_n/(2+\beta_n)},
\end{align*}
yielding the following uniqueness result.

\begin{corollary}\label{corollary1.1}
Assume that $\alpha\le \alpha_n$ and let $a_1,a_2$ satisfy \textbf{(a1)} and \textbf{(a2)} together with $a_1(0)=a_2(0)$. If $\tilde{\Lambda}^t_{a_1}=\tilde{\Lambda}^t_{a_2}$ in a neighborhood of the origin, for each $t\in \mathbb{R}$, then $a_1=a_2$.
\end{corollary}

We point out that from the proof of Theorem \ref{theorem1.1}, the knowledge of $\tilde{\Lambda}^t_a$ in a neighborhood of the origin (or equivalently the knowledge of $\tilde{\Lambda}_a$ in a neighborhood of $\mathfrak{f}_t$) only determines uniquely $a(t)$.

There is a large recent literature devoted to the uniqueness issue concerning the determination of nonlinearities in quasilinear and semilinear elliptic equations by boundary measurements. We refer for instance to  \cite{CFKKU, IS, IY,KN,KKU, KU1,KU,LLLS1,LLLS, MU, Su, SU} and references therein. Of course, this list of references is far to be exhaustive.

To the best of our knowledge, the stability issue was only considered in \cite{Ch22, CHY,Ki}. 

In the present work we adapt the analysis in \cite{Ch22} to improve the stability result in \cite{CHY}. Our construction of special solutions is borrowed from \cite{Ki}. These special solutions behave locally near a boundary point like the Levi's parametrix of the linearized operator near a singular point.

It is worth noticing that we have a Lipschitz stability estimate in the quasilinear case. While the stability estimate in the semilinear case is only of H\"older type. The fact that we have a better stability estimate in the quasilinear case can be roughtly explained by the fact the nonlinearity has more influence on the solution in the quasilinear case than in the semilinear case.

\section{Pointwise determination of the potential at the boundary}\label{section2}

Let $\sigma \in C^{0,\beta}(\overline{\Omega})$ satisfying 
\begin{equation}\label{e3.1}
 -\mathfrak{c}\le \sigma, \quad \|\sigma\|_{C^{0,\beta}(\overline{\Omega})}\le  \mathfrak{c}', 
 \end{equation}
 where $0<\beta <1$ and $\mathfrak{c}'>0$ are arbitrary fixed constants. 
 
 Let $f\in H^{1/2}(\Gamma)$. We proceed as in \cite[Lemma A2]{Ch22} in order to prove that the BVP 
\begin{equation}\label{e3.2}
\left\{
\begin{array}{ll}
-\mathrm{div}(A\nabla u) +\sigma u=0\quad \mbox{in}\; \Omega ,
\\
u_{|\Gamma}=f ,
\end{array}
\right.
\end{equation}
admits a unique weak solution $u_\sigma (f)\in H^1(\Omega)$ satisfying
\[
\|u_\sigma\|_{H^1(\Omega)}\le C\|f\|_{H^{1/2}(\Gamma)},
\]
where $C=C(n,\Omega,\kappa , \mathfrak{c},\mathfrak{c}')>0$ is a constant.

As in the preceding section,  define the Dirichlet-to-Neumann map $\Lambda_\sigma$ by
\[
\Lambda_\sigma :H^{1/2}(\Gamma)\rightarrow H^{-1/2}(\Gamma),
\]
 associated to $\sigma$, by
\[
\langle \Lambda_\sigma(f) ,\varphi\rangle_{1/2} =\int_\Omega A\nabla u_\sigma (f)\cdot \nabla vdx+\int_\Omega \sigma u_\sigma (f)vdx,\quad f,\varphi\in H^{1/2}(\Gamma),\; v\in \dot{\varphi}.
\] 
Also, define the partial Dirichlet-to-Neumann map $\tilde{\Lambda}_\sigma$ as follows
\[
\tilde{\Lambda}_\sigma:f\in H_{\Gamma_0}^{1/2}(\Gamma)\mapsto \chi\Lambda_\sigma (f)\in H^{-1/2}(\Gamma).
\]
We easily derive from the last identity that the following equality holds
\[
\langle \tilde{\Lambda}_\sigma(f) ,\varphi\rangle_{1/2} =\int_\Omega A\nabla u_\sigma (f)\cdot \nabla vdx+\int_\Omega \sigma u_\sigma (f)vdx,
\] 
for every $ f\in H^{1/2}(\Gamma)$, $\varphi\in H_{\Gamma_0}^{1/2}(\Gamma)$ and $v\in \dot{\varphi}$.

Pick $\sigma_1,\sigma_2\in C^{0,\beta}(\overline{\Omega})$ satisfying \eqref{e3.1} and  set $\sigma=\sigma_1-\sigma_2$.

For $f,g\in H^{1/2}(\Gamma)$, recall the following well known formula 
\[
\int_\Omega \sigma u_{\sigma_1}(f)u_{\sigma_2}(g)dx=\langle \Lambda_{\sigma_1}(f)-\Lambda_{\sigma_2}(f) ,g\rangle_{1/2},\quad f,g\in H^{1/2}(\Gamma).
\]
This identity yields 
\begin{equation}\label{e3.3}
\int_\Omega \sigma u_{\sigma_1}(f)u_{\sigma_2}(g)dx=\langle \tilde{\Lambda}_{\sigma_1}(f)-\tilde{\Lambda}_{\sigma_2}(f) ,g\rangle_{1/2},\quad f,g\in H_{\Gamma_0}^{1/2}(\Gamma).
\end{equation}
The Levi's parametrix associated to the operator $\mathrm{div}(A\nabla \cdot\, )$ is usually given by
\[
H(x,y)=\frac{[A^{-1}(x-y)\cdot (x-y)]^{(2-n)/2}}{(n-2)|\mathbb{S}^{n-1}|[\mathrm{det}\, A]^{1/2}},\quad x,y\in \mathbb{R}^n,\; x\ne y.
\]

Fix $x_0\in \Gamma_0$ and $r_0>0$ sufficiently small in such a way that $B(x_0,r_0)\cap \Gamma\subset \Gamma_0$. As $B(x_0,r_0) \setminus \overline{\Omega}$ contains a cone with a vertex at $x_0$, we find $\delta_0>0$ and a vector $\xi \in \mathbb{S}^{n-1}$ such that, for each $0<\delta \le \delta_0$, we have $y_\delta=x_0+\delta \xi \in B(x_0,r_0)\setminus \overline{\Omega}$ and 
\[
\mathrm{dist}(y_\delta,\overline{\Omega})\ge c\delta, \quad \mathrm{dist}(y_\delta,\partial \Omega_0)\ge r_0/2,
\]
for some constant $c=c(\Omega)>0$. 

Let $\Omega_0=\Omega\cup B(x_0,r_0)$ and fix $1\le j\le n$, $0<\delta\le \delta_0$ arbitrarily.

Set $\mathfrak{h}_\delta^j=\partial_jH(\cdot ,y_\delta)$ and let $v_\delta^j$ denotes the unique weak solution of the BVP
\[
\left\{
\begin{array}{ll}
\mathrm{div}(A\nabla v)=0\quad \mathrm{in}\; \Omega_0, 
\\
v_{|\partial \Omega_0}=\mathfrak{h}_\delta^j.
\end{array}
\right.
\]
We proceed as in \cite{Ch22} in order to derive the following estimate
\begin{equation}\label{e3.4}
\|v_\delta^j\|_{H^1(\Omega)}\le C,
\end{equation}
where $C=C(n,\Omega, \kappa, x_0)$.

Set $f_\delta^j=(\mathfrak{h}_\delta^j- v_\delta^j)_{|\Gamma}$. This definition guarantees that $f_\delta^j\in H_{\Gamma_0}^{1/2}(\Gamma)$. Let then $w_{\sigma,\delta}^j=u_\sigma(f_\delta^j)$. That is $w_{\sigma,\delta}^j$ is the weak solution of the BVP \eqref{e3.2} with $f=f_\delta^j$.

Define $z_{\sigma,\delta}^j=w_{\sigma,\delta}^j-\mathfrak{h}_\delta^j$. Then one can  check that $z_{\sigma,\delta}^j$ is the solution of the BVP
\[
\left\{
\begin{array}{ll}
-\mathrm{div}(A\nabla z)+\sigma z=-\sigma \mathfrak{h}_\delta^j \quad \mathrm{in}\; \Omega, 
\\
z_{|\Gamma}=- v_\delta^j{_{|\Gamma}}.
\end{array}
\right.
\]

The following inequality will be useful in the sequel
\[
\|\mathfrak{h}_\delta^j\|_{L^{2n/(n+2)}(\Omega)}\le C\delta^{2-n/2}.
\]
Here and henceforth, $C=C(n,\Omega, \kappa, \mathfrak{c},\mathfrak{c}',x_0)>0$ is a generic constant.

Using this inequality and the fact $H_0^1(\Omega)$ is continuously embedded in $L^{2n/(n-2)}(\Omega)$, we obtain by applying H\"older's inequality
\begin{equation}\label{e3.5}
\left|\int_\Omega \sigma \mathfrak{h}_\delta^jvdx\right|\le C\delta^{2-n/2}\|v\|_{H_0^1(\Omega)},\quad v\in H_0^1(\Omega).
\end{equation}

Using \eqref{e3.4} and \eqref{e3.5}, we  derive in a straightforward manner that the following inequality holds
\begin{equation}\label{e3.6}
\|z_{\sigma,\delta}^j\|_{H^1(\Omega)}\le C\delta^{2-n/2}.
\end{equation}

Also, we have 
\begin{equation}\label{e3.7}
\|f_\delta^j\|_{H^{1/2}(\Gamma)}\le C\delta^{-n/2}.
\end{equation}

Assume that $|\sigma (x_0)|=\sigma (x_0)$. Since $w_{\sigma_\ell,\delta}^j=z_{\sigma_\ell,\delta}^j+ \mathfrak{h}_\delta^j$, $\ell=1,2$, we get by using \eqref{e3.7}
\[
C\int_\Omega \sigma (\mathfrak{h}_\delta^j)^2dx \le \int_\Omega \sigma w_{\sigma_1,\delta}^jw_{\sigma_2,\delta}^j dx +C\delta^{4-n}.
\]
Hence
\begin{equation}\label{e3.8}
C\int_\Omega \sigma |\nabla H(\cdot,y_\delta )|^2dx \le \sum_{j=1}^n\int_\Omega \sigma w_{\sigma_1,\delta}^jw_{\sigma_2,\delta}^j dx +\delta^{4-n}.
\end{equation}
On the other hand, we find by applying \eqref{e3.7}
\begin{equation}\label{e3.9}
\left|\langle (\tilde{\Lambda}_{\sigma_1}-\tilde{\Lambda}_{\sigma_2})(f_\delta^j),f_\delta^j\rangle_{1/2}\right|\le C\delta^{-n}\|\tilde{\Lambda}_{\sigma_1}-\tilde{\Lambda}_{\sigma_2}\|_{\mathrm{op}}.
\end{equation}
 From the proof of \cite[(2.8)]{Ch21},  we obtain
 \begin{equation}\label{e3.10}
C |\sigma (x_0)|\le \delta^{n-2}\int_\Omega \sigma |\nabla H(\cdot,y_\delta )|^2dx+\delta^\beta.
 \end{equation}

In light of \eqref{e3.3}, we get by putting together \eqref{e3.8}, \eqref{e3.9} and \eqref{e3.10}
\[
C |\sigma (x_0)|\le \delta^{-2}\|\tilde{\Lambda}_{\sigma_1}-\tilde{\Lambda}_{\sigma_2}\|_{\mathrm{op}}+\delta^{\beta},
\]
from which we derive in a straightforward manner
\begin{equation}\label{e3.11}
|\sigma (x_0)|\le C\|\tilde{\Lambda}_{\sigma_1}-\tilde{\Lambda}_{\sigma_2}\|_{\mathrm{op}}^{\beta/(2+\beta)}.
\end{equation}

\begin{remark}
{\rm
Assume that $\Gamma_0=\Gamma_1=\Gamma$. If $x_0$ in \eqref{e3.11} is chosen so that $\sigma(x_0)=\|\sigma\|_{C(\Gamma)}$ then we get 
\[
\|\sigma_1-\sigma_2\|_{C(\Gamma)}\le C\|\Lambda_{\sigma_1}-\Lambda_{\sigma_2}\|_{\mathrm{op}}^{\beta/(2+\beta)}.
\]
In fact this estimate is not optimal since we know that in the linear case we have a Lipschitz stability (e.g. \cite[Theorem 4.2]{Ch21}).
}
\end{remark}

\section{Proof of Theorem \ref{theorem1.1}}\label{section3}

Before we proceed to the proof of Theorem \ref{theorem1.1} we establish some preliminary results. First,  mimicking the proof of \cite[Lemma 2.2, Corollary 3.2 and Lemma 3.1]{CHY}, we get

\begin{proposition}\label{proposition2.1}
Assume that $\alpha \le \alpha_n$ and $a$ satisfies \textbf{(a1)}. If $f\in \mathfrak{Z}_n$ then $u_a(f)\in \mathfrak{X}_n$. Furthermore, we have
\begin{equation}\label{e2.1}
\|u_a(f)\|_{\mathfrak{X}_n}\le \mathcal{C}_n(1+m+m^\alpha),\quad f\in B_{\mathfrak{Z}_n}(m).
\end{equation}
\end{proposition}

In the case $a=0$, we have instead of \eqref{e2.1} the following estimate
\begin{equation}\label{e2.2}
\|u_0(f)\|_{\mathfrak{X}_n}\le C_0\|f\|_{\mathfrak{Z}_n},\quad f\in \mathfrak{Z}_n,
\end{equation}
where $C_0=C_0(n,\Omega,\kappa)$.

Let $w_a(f)=u_a(f)-u_0(f)$, where $f\in \mathfrak{Z}_n$. Then $w_a(f)\in H_0^1(\Omega)$ and we have
\begin{equation}\label{e2.3}
\int_\Omega A\nabla w_a(f)\cdot\nabla vdx+\int_\Omega a(u_a(f))vdx=0,\quad v \in H_0^1(\Omega ).
\end{equation}

\begin{lemma}\label{lemma2.1}
Assume that $\alpha \le \alpha_n$ and $a$ satisfies \textbf{(a1)} and \textbf{(a2)}. Then
\begin{equation}\label{e2.4}
\|u_a(f)-u_a(g)\|_{H^1(\Omega)}\le C_m\|f-g\|_{\mathfrak{Z}_n},\quad f,g\in B_{\mathfrak{Z}_n}(m).
\end{equation}
\end{lemma}

\begin{proof}

Let $f,g\in \mathfrak{Z}_n$. Set $y_a=w_a(f)-w_a(g)$ and
\[
\mathbf{d}=\int_0^1a'(u_a(f)+t(u_a(f)-u_a(g)))dt.
\]
Then we have
\[
a(u_a(f))-a(u_a(g))=\mathbf{d}(u_a(f)-u_a(g))= \mathbf{d}y_a-\mathbf{d}u_0(f-g).
\]
This, together with \eqref{e2.3} for both $f$ and $g$ yield in a straightforward manner
\begin{equation}\label{e2.5}
\int_\Omega A\nabla y_a(f)\cdot\nabla vdx+\int_\Omega \mathbf{d}y_avdx=\int_\Omega \mathbf{d}u_0(f-g)vdx,\quad v \in H_0^1(\Omega ).
\end{equation}
Since $y_a \in H_0^1(\Omega)$, $v=y_a$ in \eqref{e2.5} gives
\[
\int_\Omega A\nabla y_a(f)\cdot\nabla y_adx+\int_\Omega \mathbf{d}y_a^2dx=\int_\Omega \mathbf{d}u_0(f-g)y_adx.
\]
Hence
\[
(\kappa -\mathfrak{c}\lambda_1^{-1})\|\nabla y_a\|_{L^2(\Omega)}\le \gamma(\varrho_m)\lambda_1^{-1/2}\|u_0(f-g)\|_{L^2(\Omega)}.
\]
That is we have
\[
\|\nabla y_a\|_{L^2(\Omega)}\le C_m\|u_0(f-g)\|_{L^2(\Omega)}.
\]
In consequence we obtain
\[
\|\nabla w_a\|_{L^2(\Omega)}\le C_m\|u_0(f-g)\|_{L^2(\Omega)}.
\]
This inequality, combined with \eqref{e2.2}, implies
\[
\|u_a(f)-u_a(g)\|_{H^1(\Omega)}\le C_m\|f-g\|_{\mathfrak{Z}_n}.
\]
This is the expected inequality.
\end{proof}

Assume that the assumptions of the preceding lemma hold and let $f\in B_{\mathfrak{Z}_n}(m)$. Define on $H_0^1(\Omega)\times H_0^1(\Omega)$ the continuous bilinear form
\[
\mathfrak{b}(u,v)=\int_\Omega[A\nabla u\cdot \nabla v+a'(u_a(f))uv]dx,\quad u,v\in H_0^1(\Omega).
\]
Let $h\in \mathfrak{Z}_m$. As $\mathfrak{b}$ is coercive, according to Lax-Milgram's lemma the variational problem
\[
\mathfrak{b}(u,v)=-\int_\Omega a'(u_a(f))u_0(h)v,\quad v\in H_0^1(\Omega),
\]
admits a unique solution $\ell_a(f,h)\in H_0^1(\Omega)$. Similar computations as above give
\begin{equation}\label{e2.6}
\|\ell_a(f,h)\|_{H^1(\Omega)}\le C_m\|h\|_{\mathfrak{Z}_n}.
\end{equation}

Let $f\in B_{\mathfrak{Z}_n}(m)$ and $h\in \mathfrak{Z}_n$ sufficiently small in such a way that $f+h\in B_{\mathfrak{Z}_n}(m)$. Set
\begin{align*}
&R=w_a(f+h)-w_a(f)-\ell_a(f,h),
\\
&K=-(u_a(f+h)-u_a(f))\int_\Omega [a'(u_a(f)+t(u_a(f+h)-u_a(f))-a'(u_a(f))]dt.
\end{align*}
We easily prove that $R$ satisfies
\[
\mathfrak{b}(R,v)=\int_\Omega Kvdx,\quad v\in H_0^1(\Omega).
\]
Hence
\[
\|R\|_{H^1(\Omega)}\le C_m\|K\|_{L^2(\Omega)}.
\]
Using the uniform continuity of $a'$ in $[-\varrho_m,\varrho_m]$ and \eqref{e2.4}, we get $\|K\|_{L^2(\Omega)}=o\left(\|h\|_{\mathfrak{Z}_n}\right)$. In other words, we proved that $w_a$ is Fr\'echet differentiable at $f$ and $dw_a(f)(h)=\ell_a(f,h)$. Therefore, $u_a$ is Fr\'echet differentiable at $f$ and 
\[
du_a(f)(h)=\ell_a(f,h)+u_0(h).
\]
From the definition of $\ell_a(f,h)$ and $u_0(h)$, we see that $du_a(f)(h)$ is the weak solution of the BVP
\begin{equation}\label{e2.7}
\left\{
\begin{array}{ll}
-\mathrm{div}(A\nabla u) +a'(u_a(f))u=0\quad \mbox{in}\; \Omega ,
\\
u_{|\Gamma}=h .
\end{array}
\right.
\end{equation}
In addition the following estimate holds
\begin{equation}\label{e2.8}
\| du_a(f)(h)\|_{H^1(\Omega)}\le C_m\|h\|_{H^{1/2}(\Gamma)},\quad f\in B_{\mathfrak{Z}_n}(m),\; h\in \mathfrak{Z}_n.
\end{equation}

Let $t\in \mathbb{R}$. Using its definition, we obtain that $\tilde{\Lambda}_a$ is Fr\'echet differentiable in a neighborhood of the origin with
\[
\langle d\tilde{\Lambda}_a^t (0)(h),\varphi\rangle =\int_\Omega [A\nabla v_a(\mathfrak{f}_t,h)\cdot \nabla w+\sigma_a^t v_a(\mathfrak{f}_t,h)w]dx, 
\]
for every $\varphi \in H_{\Gamma_0}^{1/2}(\Gamma)$, $w\in  \dot{\varphi}$ and $h\in \mathfrak{Z}_n^0$, where $\sigma_a^t=a'(u_a(\mathfrak{f}_t))$ and $v_a(\mathfrak{f}_t,h)$ is the solution of the BVP \eqref{e2.7} with $f=\mathfrak{f}_t$.

Recall that $m_n^\tau =\max_{|t|\le \tau}\|\mathfrak{f}_t\|_{\mathfrak{Z}_n^0}$. Then \eqref{e2.8} gives
\[
\|v_a(\mathfrak{f}_t,h)\|_{H^1(\Omega)}\le C_{m_n^\tau}\|h\|_{H^{1/2}(\Gamma)},\quad h\in \mathfrak{Z}_n^0,\; |t|\le \tau.
\]
Hence $d\tilde{\Lambda}_a^t (0)$ is extended to bounded operator from $H_{\Gamma_0}^{1/2}(\Gamma)$ into $H^{1/2}(\Gamma)$ and we have
\[
\sup_{|t|\le \tau}\|d\tilde{\Lambda}_a^t (0)\|_{\mathrm{op}}\le C_{m_n^\tau}.
\]

We are now ready to complete the proof of Theorem \ref{theorem1.1}. Let then $a_1,a_2$ be as in the statement of Theorem \ref{theorem1.1}. Modifying slightly the proof of  \cite[Lemma 5.2]{CHY}, we show that $\sigma_{a_j}^t\in C^{0,\beta_n}(\overline{\Omega})$ with
\[
\|\sigma_{a_j}^t\|_{C^{0,\beta_n}(\overline{\Omega})}\le C_{m_\tau^n},\quad |t|\le \tau,\; j=1,2.
\]
Taking into account that $\sigma_{a_j}^t(x_\ast)=a_j'(t)$, we get by applying \eqref{e3.11} with $x_0=x_\ast$, the following inequality
\[
|a'_1(t)-a'_2(t)|\le C_{m_n^\tau}\|d\tilde{\Lambda}_{a_1}^t (0)-d\tilde{\Lambda}_{a_2}^t (0)\|_{\mathrm{op}}^{\beta_n/(2+\beta_n)},\quad |t|\le \tau,
\]
and hence
\[
\|a'_1-a'_2\|_{C([-\tau ,\tau])}\le  C_{m_n^\tau}\sup_{|t|\le \tau}\|d\tilde{\Lambda}_{a_1}^t (0)-d\tilde{\Lambda}_{a_2}^t (0)\|_{\mathrm{op}}^{\beta_n/(2+\beta_n)}.
\]
This is the expected inequality.

\end{document}